\documentclass[11pt]{amsart}

\usepackage[T1]{fontenc}
\usepackage{lmodern}
\usepackage{amsmath,amssymb,amsthm,mathtools}
\usepackage{microtype}
\usepackage[hidelinks]{hyperref}

\newtheorem{theorem}{Theorem}[section]
\newtheorem{proposition}[theorem]{Proposition}

\newtheorem{corollary}[theorem]{Corollary}

\theoremstyle{definition}
\newtheorem{definition}[theorem]{Definition}
\newtheorem{question}[theorem]{Question}

\theoremstyle{remark}
\newtheorem{remark}[theorem]{Remark}

\newcommand{\br}{\operatorname{Br}}
\newcommand{\Tlink}{T}

\title[Infinite Families of Generalised T-Link Presentations]{Every Link Has Infinitely Many Explicit Generalised T-Link Presentations}

\author{Thiago de Paiva}
\address[]{Beijing International Center for Mathematical Research, Peking University, Beijing 100871, China P.R.}
\email[]{thhiagodepaiva@gmail.com}
\author{Yi Liu}
\address{Beijing International Center for Mathematical Research, Peking University, Beijing 100871, China P.R.}
\email{liuyi@bicmr.pku.edu.cn}

\begin{document}

\begin{abstract}
Generalised $T$-links provide a simple description of all links in
$S^3$ as closures of products of standard twisting blocks,
parametrised by finite sequences of integers. We prove that every
link admits infinitely many pairwise distinct generalised $T$-link
presentations. Starting from any such presentation, we give explicit
parameter transformations that preserve the represented link and
generate families of pairwise distinct presentations depending on
arbitrarily many independent integer parameters.
\end{abstract}

\maketitle

\section{Introduction}
\label{sec:introduction}

Birman and Kofman introduced $T$-links as closures of products of
positive twisting blocks supported on increasing numbers of strands
\cite{BirmanKofman2009}. Generalised $T$-links extend this construction
by allowing the final twisting exponent to be negative and including
an offset that records additional split unknotted components.
De Paiva, Hui and Rodr\'{\i}guez Migueles proved that every link in
$S^3$ admits a generalised $T$-link presentation
\cite{dePaivaHuiRodriguez2024}. This gives a universal parametrisation
of links by finite strings of integers subject to explicit ordering
and sign conditions.

The universality theorem leads naturally to the study of equivalence
between parameter strings: how can different parameters describe the
same ambient isotopy class of links? In this paper, we construct
explicit families of such equivalences. Starting from any generalised
$T$-link presentation, we produce infinitely many pairwise distinct
presentations of the same link.  
To state the result, write
\[
   \delta_{m,d}=\sigma_{d+1}\cdots\sigma_{d+m-1},
   \qquad
   \delta_m=\delta_{m,0},
\]
for integers $m\geq2$ and $d\geq0$, where the $\sigma_i$ are the
standard braid generators. A generalised $T$-link presentation is a
parameter string
\[
   \mathcal P=((r_1,s_1),\ldots,(r_\ell,s_\ell),d)
\]
representing the closure of the braid
\[
   \delta_{r_1,d}^{s_1}\cdots\delta_{r_\ell,d}^{s_\ell}
\]
on $d+r_\ell$ strands. Here $\ell\geq1$, $d\geq0$,
$2\leq r_1<\cdots<r_\ell$, the exponents $s_i$ are positive for
$i<\ell$, and $s_\ell\in\mathbb Z\setminus\{-1,0,1\}$.
We recall the full definition in Section~\ref{sec:background}.
Presentations are regarded as distinct when their parameter strings
differ, even if their braid closures are ambient isotopic.

Our main theorem gives an explicit formula for these families of
generalised $T$-link presentations.

\begin{theorem}
\label{thm:main-intro}
Let
\[
   \mathcal P=((r_1,s_1),\ldots,(r_\ell,s_\ell),d)
\]
be any generalised $T$-link presentation, and let $L$ be the link
represented by $\mathcal P$. Put
\[
   R_0=r_\ell,
   \qquad
   s=s_\ell,
\]
and define
\[
   k_0=\max\left\{
      2,\left\lfloor\frac{-s}{R_0}\right\rfloor+1
   \right\}.
\]

For every integer $n\geq1$ and every finite sequence of integers
\[
   \mathbf{k}=(k_1,\ldots,k_n)
\]
satisfying
\[
   k_1\geq k_0,
   \qquad
   k_i\geq2\quad(2\leq i\leq n),
\]
define
\[
   R_i=k_iR_{i-1},
   \qquad 1\leq i\leq n,
\]
and set
\[
\begin{aligned}
   \mathcal P_{\mathbf{k}}=\bigl(
      &(r_1,s_1),\ldots,(r_{\ell-1},s_{\ell-1}),\\
      &(R_0,s+R_1),
      \bigl(R_i,R_{i+1}-R_{i-1}\bigr)_{i=1}^{n-1},\\
      &(R_n,-R_{n-1}),d
   \bigr),
\end{aligned}
\]
where the initial list is empty when $\ell=1$, and the indexed
middle list is empty when $n=1$.

Then the following statements hold.
\begin{enumerate}
\item
Every $\mathcal P_{\mathbf{k}}$ is a generalised $T$-link
presentation representing $L$.

\item
The presentations $\mathcal P_{\mathbf{k}}$, as $\mathbf{k}$ ranges
over all finite sequences satisfying the preceding conditions,
are pairwise distinct.
\end{enumerate}

Consequently, every link in $S^3$ admits infinitely many pairwise
distinct generalised $T$-link presentations.
\end{theorem}

For each fixed $n$, the theorem allows the $n$ integers
$k_1,\ldots,k_n$ to be chosen independently, subject only to the
stated lower bounds, and different choices give different parameter
strings. Moreover, $n$ itself is arbitrary. Thus every link admits
explicit families with arbitrarily many independently chosen integer
parameters. The construction retains the first $\ell-1$ twisting
pairs and the offset of the initial presentation, while increasing
the number of twisting pairs from $\ell$ to $\ell+n$.

Since generalised $T$-links represent all links in $S^3$, understanding
when two different parameter strings determine the same link is a
natural part of the classification problem for this universal
parametrisation. Theorem~\ref{thm:main-intro} identifies a large and
explicit source of non-uniqueness: every link admits families of
pairwise distinct presentations depending on arbitrarily many
independent integer parameters. 

This naturally suggests going one step further. Theorem~\ref{thm:main-intro} identifies an explicit family of equivalences among generalised \(T\)-link presentations. One may therefore ask what remains of the non-uniqueness after presentations related by these equivalences are identified.

\begin{question}
\label{question:reduced-multiplicity}
If we disregard the repetitions arising from the transformations
constructed in this paper, does every link in the 3-sphere still admit
infinitely many distinct generalised $T$-link presentations?
\end{question}

More broadly, determining the equivalence relations among generalised
$T$-link parameters may provide a useful approach to understanding
the classification of this universal family. Theorem~\ref{thm:main-intro}
gives one explicit family of such relations, while
Question~\ref{question:reduced-multiplicity} asks whether the
redundancy of the parametrisation persists even after these relations
have been taken into account.

The paper is organised as follows. Section~\ref{sec:background}
recalls generalised $T$-links and their universality.
Section~\ref{sec:negative-reduction} establishes the negative
braid-reduction isotopy and the cancellation identity underlying our
construction. In Section~\ref{sec:infinite-representations}, we first
prove the construction for a single step and then obtain
Theorem~\ref{thm:main-intro} by iteration.

\section*{Acknowledgements}

The authors thank Pierre Dehornoy for a question posed to the first
author during the workshop \emph{Dynamics, Foliations, and Geometry III},
held at MATRIX in Australia, which helped motivate this work.

The authors used a large language model to generate revised wording for portions of the manuscript and improve its exposition. The authors take full intellectual responsibility for the entire content, including all mathematical statements, proofs, and references.

\section{Generalised \texorpdfstring{$T$}{T}-links}
\label{sec:background}

In this section, we recall the definition of generalised $T$-links and
the universality result used in the proof of our main theorem. We also
fix the braid notation used throughout the paper.

Let $\br_m$ denote the braid group on $m$ strands with standard
generators $\sigma_1,\ldots,\sigma_{m-1}$. For integers $m\geq2$ and
$d\geq0$, set
\[
   \delta_{m,d}
   =\sigma_{d+1}\sigma_{d+2}\cdots\sigma_{d+m-1},
   \qquad
   \delta_m=\delta_{m,0}.
\]
The braid $\delta_{m,d}\in\br_{d+m}$ is supported on strands
$d+1,\ldots,d+m$, and its first $d$ strands are trivial. 

We denote the closure of a
braid $\beta$ by $\widehat{\beta}$ and use $\cong$ to denote ambient
isotopy of links in $S^3$.

We recall the definition of generalised $T$-links from
\cite[Definition~3.8]{dePaivaHuiRodriguez2024}.

\begin{definition}
\label{def:generalised-t-link}
Let $\ell\geq1$ and $d\geq0$ be integers, and let
$r_1,\ldots,r_\ell$ be integers satisfying
\[
   2\leq r_1<\cdots<r_\ell.
\]
Let $s_1,\ldots,s_{\ell-1}$ be positive integers, and let
\[
   s_\ell\in\mathbb Z\setminus\{-1,0,1\}.
\]
A \emph{generalised $T$-link presentation} is a parameter string
\[
   \mathcal P=((r_1,s_1),\ldots,(r_\ell,s_\ell),d).
\]
It represents the generalised $T$-link
\[
   \Tlink((r_1,s_1),\ldots,(r_\ell,s_\ell),d),
\]
defined as the closure of the braid
\[
   \delta_{r_1,d}^{s_1}\cdots\delta_{r_\ell,d}^{s_\ell}
   \in\br_{d+r_\ell}.
\]
The integers $r_\ell$ and $s_\ell$ are called the \emph{final
strand parameter} and the \emph{final twisting exponent},
respectively. When $\ell=1$, the list preceding the final pair
is empty.
\end{definition}

Two presentations are regarded as distinct when their parameter
strings differ, even if the links they represent are ambient isotopic.

When $d=0$ and $s_\ell>0$, the defining braid belongs to the ordinary
$T$-link family of Birman and Kofman \cite{BirmanKofman2009}.
Generalised $T$-links allow the final twisting exponent to be negative
while requiring all preceding twisting exponents to remain positive.

The offset $d$ records $d$ additional split unknotted components.
Indeed, the first $d$ strands do not meet any generator in the defining
braid, and their closures form a split union of $d$ unknots. Thus the
represented link is the split union of these components with the
closure of
\[
   \delta_{r_1}^{s_1}\cdots\delta_{r_\ell}^{s_\ell}
   \in\br_{r_\ell}.
\]
In particular, any presentation of a knot or a nonsplit link must
have $d=0$.

The restriction on the final exponent does not exclude unlinks.
Indeed, $\Tlink((2,2),(3,-2),d)$ represents the unlink with $d+1$
components. Using the Artin braid relation
\[
   \sigma_2^{-1}\sigma_1^{-1}\sigma_2^{-1}
   =
   \sigma_1^{-1}\sigma_2^{-1}\sigma_1^{-1},
\]
we obtain
\[
\begin{aligned}
   \sigma_1^{-1}
   \bigl(\delta_2^2\delta_3^{-2}\bigr)
   \sigma_1
   &=
   \sigma_1\sigma_2^{-1}\sigma_1^{-1}\sigma_2^{-1}\\
   &=
   \sigma_2^{-1}\sigma_1^{-1}.
\end{aligned}
\]
After a cyclic permutation of the factors, two negative Markov
destabilisations reduce the braid on the right to the trivial
one-strand braid. Its closure is therefore the unknot, while the
offset contributes $d$ additional split unknotted components.

We will use the following universality result
\cite[Corollary~3.13]{dePaivaHuiRodriguez2024}.

\begin{theorem}
\label{thm:generalised-t-universality}
Every link in $S^3$ admits a generalised $T$-link presentation.
\end{theorem}

This theorem provides an initial presentation for every link type.
In Section~\ref{sec:infinite-representations}, we construct, from any
such presentation, explicit infinite families of pairwise distinct
generalised $T$-link presentations representing the same link.

\section{A negative braid-reduction isotopy}
\label{sec:negative-reduction}

The construction of infinitely many presentations in the next section
relies on a cancellation between a positive twisting block and a negative
twisting block supported on a larger number of strands. To establish this
cancellation, we introduce a braid-reduction isotopy that removes the
additional strands one at a time. At each step, the effect of the removed
strand is recorded by a negative partial twist on the first $q$ strands.

The strand-removal step is the negative counterpart of the positive
move described in \cite[Proposition~2.1 and Figure~1]{dePaiva2024} or \cite[Proposition~3.2]{Lorenzknots}.

Recall that
\[
   \delta_m=\sigma_1\cdots\sigma_{m-1},
\]
and set
\[
   \overline{\delta}_m=\sigma_{m-1}\cdots\sigma_1.
\]
The standard full-twist identities are
\[
   \delta_m^m=\overline{\delta}_m^m=\Delta_m^2.
\]
Here $\Delta_m$ denotes the Garside half-twist on $m$ strands;
its square is the full twist and is central in $\br_m$.

\begin{proposition}
\label{prop:negative-reduction}
Let $p,q,r$ be integers satisfying
\[
   1<q\leq r<p,
\]
and let $\beta\in\br_r$, regarded as an element of $\br_p$ by adding
$p-r$ trivial strands on the right. Then
\[
   \widehat{\beta\delta_p^{-q}}
   \cong
   \widehat{
      (\sigma_1^{-1}\sigma_2^{-1}\cdots
      \sigma_{q-1}^{-1})^{p-r}
      \beta\delta_r^{-q}
   }.
\]
\end{proposition}

\begin{proof}
For brevity, set
\[
   A_q=\sigma_1^{-1}\sigma_2^{-1}\cdots\sigma_{q-1}^{-1}.
\]

We first describe an isotopy that reduces the number of strands by one.
Let $m>r$ be an integer and regard $\gamma\in\br_r$ as an element of
$\br_m$ by adding trivial strands on the right. We claim that
\begin{equation}
\label{eq:one-step-negative}
   \widehat{\gamma\delta_m^{-q}}
   \cong
   \widehat{A_q\gamma\delta_{m-1}^{-q}}.
\end{equation}

Draw $\gamma$ above $\delta_m^{-q}$. Starting at position $m$ at the top of $\delta_m^{-q}$,
follow the strand upwards through $\gamma$. Since $m>r$, this part
of the strand is vertical. Continue along the braid closure to
position $m$ at the bottom of $\delta_m^{-q}$. Traversing the
negative twisting block upwards, the strand passes under one strand
in each layer and reaches position $m-q$ at the top of the block.
Thus, the $m$th and $(m-q)$th top positions of $\delta_m^{-q}$ are
connected by an underpassing arc that travels once around the braid
closure.

Push this long arc down and shrink it beneath the remaining braid.
This removes the $m$th strand and changes the negative twisting block
$\delta_m^{-q}$ into $\delta_{m-1}^{-q}$. The shortened strand passes
under the last $q-1$ strands involved in the reduction and produces,
in order, the negative crossings
\[
   C_{m,q}
   =\sigma_{m-q}^{-1}\sigma_{m-q+1}^{-1}
      \cdots\sigma_{m-2}^{-1}.
\]
Thus, after this step, the resulting closed braid is represented by
the $(m-1)$-braid
\[
   \gamma
   \bigl(
      \sigma_{m-q}^{-1}\sigma_{m-q+1}^{-1}
      \cdots\sigma_{m-2}^{-1}
   \bigr)
   \delta_{m-1}^{-q}
   =\gamma C_{m,q}\delta_{m-1}^{-q}.
\]

Next, slide the shortened strand beneath the remaining twisting
block and pull it around the braid closure to the top of the
diagram. This transports the partial twist from the last $q$
strands to the first $q$ strands. We verify the order of the
resulting crossings algebraically. Put $n=m-1$. The Artin braid
relations give
\[
   \delta_n\sigma_i^{-1}\delta_n^{-1}
   =\sigma_{i+1}^{-1},
   \qquad 1\leq i\leq n-2.
\]
Since $q\leq n$, it follows that
\[
   C_{m,q}
   =\delta_n^{n-q}A_q\delta_n^{-(n-q)}.
\]
Using the centrality of $\delta_n^n$, we obtain
\[
   \delta_n^q C_{m,q}\delta_n^{-q}
   =\delta_n^n A_q\delta_n^{-n}
   =A_q.
\]
Equivalently,
\[
   C_{m,q}\delta_n^{-q}=\delta_n^{-q}A_q.
\]
Consequently, moving this partial twist through the twisting block
and then cyclically permuting the factors inside the braid closure
gives
\[
   \widehat{\gamma C_{m,q}\delta_n^{-q}}
   =\widehat{\gamma\delta_n^{-q}A_q}
   \cong\widehat{A_q\gamma\delta_n^{-q}}.
\]
This proves~\eqref{eq:one-step-negative}.

We now apply~\eqref{eq:one-step-negative} repeatedly. Since $q\leq r$,
the braid $A_q$ is supported on the first $r$ strands. Hence, after
each reduction, the braid preceding the negative twisting block
remains an $r$-braid, and the same isotopy can be applied again.
Therefore,
\[
   \widehat{\beta\delta_p^{-q}}
   \cong
   \widehat{A_q^j\beta\delta_{p-j}^{-q}},
   \qquad 0\leq j\leq p-r.
\]
Taking $j=p-r$ and substituting the definition of $A_q$ proves the
proposition.
\end{proof}

The special case $q=r$ provides the mechanism needed in the construction
of infinitely many presentations.

\begin{corollary}
\label{cor:negative-full-twist}
Let $r\geq2$ and $p>r$ be integers, and let $\beta\in\br_r$.
Then
\[
   \widehat{\beta\delta_p^{-r}}
   \cong
   \widehat{
      \beta\overline{\delta}_r^{-(p-r)}\delta_r^{-r}
   }.
\]
If $p=kr$ for some integer $k\geq2$, then
\[
   \widehat{\beta\delta_{kr}^{-r}}
   \cong
   \widehat{\beta\delta_r^{-kr}}.
\]
\end{corollary}

\begin{proof}
Putting $q=r$ in Proposition~\ref{prop:negative-reduction} gives
\[
   \widehat{\beta\delta_p^{-r}}
   \cong
   \widehat{
      \overline{\delta}_r^{-(p-r)}\beta\delta_r^{-r}
   },
\]
because
\[
   \sigma_1^{-1}\sigma_2^{-1}\cdots\sigma_{r-1}^{-1}
   =\overline{\delta}_r^{-1}.
\]
A cyclic permutation of the factors inside the braid closure gives
\[
   \widehat{
      \overline{\delta}_r^{-(p-r)}\beta\delta_r^{-r}
   }
   \cong
   \widehat{
      \beta\delta_r^{-r}\overline{\delta}_r^{-(p-r)}
   }.
\]
Since $\delta_r^{-r}=\Delta_r^{-2}$ is central in $\br_r$, the
latter braid equals
\[
   \beta\overline{\delta}_r^{-(p-r)}\delta_r^{-r}.
\]
This proves the first assertion.

Now suppose that $p=kr$. Since $p-r=(k-1)r$, we have
\[
\begin{aligned}
   \overline{\delta}_r^{-(p-r)}\delta_r^{-r}
   &=
   \bigl(\overline{\delta}_r^r\bigr)^{-(k-1)}
   \bigl(\delta_r^r\bigr)^{-1}\\
   &=
   \bigl(\Delta_r^2\bigr)^{-(k-1)}
   \bigl(\Delta_r^2\bigr)^{-1}\\
   &=
   \bigl(\Delta_r^2\bigr)^{-k}\\
   &=
   \delta_r^{-kr}.
\end{aligned}
\]
The second assertion follows.
\end{proof}

\begin{corollary}
\label{cor:cancellation}
Let $r\geq2$ be an integer and let $B\in\br_r$. For every integer
$k\geq2$,
\[
   \widehat{B\delta_r^{kr}\delta_{kr}^{-r}}
   \cong
   \widehat{B}.
\]
\end{corollary}

\begin{proof}
Apply the second assertion of
Corollary~\ref{cor:negative-full-twist} to the $r$-braid
$B\delta_r^{kr}$. We obtain
\begin{align*}
   \widehat{B\delta_r^{kr}\delta_{kr}^{-r}}
   &\cong
   \widehat{B\delta_r^{kr}\delta_r^{-kr}}\\
   &=
   \widehat{B}. \qedhere
\end{align*}
\end{proof}

Corollary~\ref{cor:cancellation} will be used in the next section
to construct infinitely many generalised $T$-link presentations
of a fixed link.

\section{Infinitely many generalised \texorpdfstring{$T$}{T}-link presentations}
\label{sec:infinite-representations}

We now use the cancellation isotopy from the preceding section to
construct explicit infinite families of generalised $T$-link
presentations of the same link. We first describe one application of
the construction and then iterate it to prove
Theorem~\ref{thm:main-intro}.

\begin{theorem}
\label{thm:explicit-family}
Let
\[
   \mathcal P=((r_1,s_1),\ldots,(r_\ell,s_\ell),d)
\]
be any generalised $T$-link presentation, and let $L$ be the link
represented by $\mathcal P$. Put
\[
   r=r_\ell,
   \qquad
   s=s_\ell,
\]
and define
\[
   k_0=\max\left\{
      2,\left\lfloor\frac{-s}{r}\right\rfloor+1
   \right\}.
\]
For every integer $k\geq k_0$, define
\[
\begin{aligned}
   \mathcal P_k=\bigl(
      &(r_1,s_1),\ldots,(r_{\ell-1},s_{\ell-1}),\\
      &(r,s+kr),(kr,-r),d
   \bigr),
\end{aligned}
\]
where the initial list is empty when $\ell=1$.

Then the following statements hold.
\begin{enumerate}
\item
Every $\mathcal P_k$ is a generalised $T$-link presentation
representing $L$.

\item
The presentations $\mathcal P_k$, with $k\geq k_0$, are pairwise
distinct.
\end{enumerate}
\end{theorem}

\begin{proof}
By the definition of $k_0$, every integer $k\geq k_0$ satisfies
$k\geq2$ and
\[
   k>\frac{-s}{r}.
\]
Hence $s+kr>0$, so the new penultimate twisting exponent is positive.
Since $r\geq2$, we also have
\[
   r<kr
   \qquad\text{and}\qquad
   -r\in\mathbb Z\setminus\{-1,0,1\}.
\]
All preceding twisting exponents remain positive, and the strand
parameters remain strictly increasing. Therefore, $\mathcal P_k$
satisfies Definition~\ref{def:generalised-t-link}. This includes
$\ell=1$, when
\[
   \mathcal P_k=((r,s+kr),(kr,-r),d).
\]

We next verify that $\mathcal P_k$ represents $L$. After omitting the
first $d$ trivial strands, the braid defining $\mathcal P$ is
\[
   B=
   \delta_{r_1}^{s_1}\cdots
   \delta_{r_{\ell-1}}^{s_{\ell-1}}
   \delta_r^s
   \in\br_r,
\]
where the product preceding $\delta_r^s$ is empty when $\ell=1$.
Regarding $B$ as a braid on $kr$ strands by adding trivial strands
on the right, the corresponding braid for $\mathcal P_k$ is
\[
\begin{aligned}
   B_k
   &=
   \delta_{r_1}^{s_1}\cdots
   \delta_{r_{\ell-1}}^{s_{\ell-1}}
   \delta_r^{s+kr}\delta_{kr}^{-r}\\
   &=B\delta_r^{kr}\delta_{kr}^{-r}.
\end{aligned}
\]
Corollary~\ref{cor:cancellation} gives
\[
   \widehat{B_k}
   =\widehat{B\delta_r^{kr}\delta_{kr}^{-r}}
   \cong\widehat{B}.
\]
Restoring the first $d$ trivial strands adds $d$ split unknotted
components to each closure. Taking the split union with these
components preserves the preceding isotopy. Thus $\mathcal P_k$
represents $L$ for every $k\geq k_0$, proving the first assertion.

Finally, the final strand parameter of $\mathcal P_k$ is $kr$.
Since $r>0$, different values of $k$ give different parameter
strings. This proves the second assertion.
\end{proof}

Theorem~\ref{thm:explicit-family} gives the case $n=1$ of
Theorem~\ref{thm:main-intro}. We now show that its iteration yields
the formula stated in the introduction for every $n\geq1$.

\begin{proof}[Proof of Theorem~\ref{thm:main-intro}]
Fix an indexing sequence
\[
   \mathbf{k}=(k_1,\ldots,k_n)
\]
satisfying the conditions of the theorem, and write
\[
   \mathbf{k}^{(m)}=(k_1,\ldots,k_m),
   \qquad 1\leq m\leq n,
\]
for its initial subsequence of length $m$. Recall that
\[
   R_0=r_\ell,
   \qquad
   s=s_\ell,
   \qquad
   R_i=k_iR_{i-1}
   \quad(1\leq i\leq n).
\]

We first prove by induction on $m$ that applying the construction
successively with parameters $k_1,\ldots,k_m$ produces the stated
generalised $T$-link presentation $\mathcal P_{\mathbf{k}^{(m)}}$,
representing $L$ and having final pair
\[
   (R_m,-R_{m-1}).
\]

For $m=1$, since $k_1\geq k_0$,
Theorem~\ref{thm:explicit-family} applies to $\mathcal P$ with
$k=k_1$. The resulting presentation is
\[
\begin{aligned}
   \mathcal P_{\mathbf{k}^{(1)}}=\bigl(
      &(r_1,s_1),\ldots,(r_{\ell-1},s_{\ell-1}),\\
      &(R_0,s+R_1),(R_1,-R_0),d
   \bigr),
\end{aligned}
\]
where $R_1=k_1R_0$. It represents $L$ and has the required final pair.

Now let $2\leq m\leq n$ and suppose that the assertion holds for
$m-1$. The final pair of
$\mathcal P_{\mathbf{k}^{(m-1)}}$ is
\[
   (R_{m-1},-R_{m-2}).
\]
Since $k_{m-1}\geq2$, we have
\[
   \frac{R_{m-2}}{R_{m-1}}
   =\frac{1}{k_{m-1}}\in(0,1).
\]
Consequently, the lower bound required by
Theorem~\ref{thm:explicit-family} at this stage is
\[
   \max\left\{
      2,\left\lfloor\frac{R_{m-2}}{R_{m-1}}\right\rfloor+1
   \right\}=2.
\]
Thus that theorem applies with any $k_m\geq2$ and replaces the final
pair by
\[
   (R_{m-1},-R_{m-2}+k_mR_{m-1}),
   \qquad
   (k_mR_{m-1},-R_{m-1}).
\]
Using $R_m=k_mR_{m-1}$, these become
\[
   (R_{m-1},R_m-R_{m-2}),
   \qquad
   (R_m,-R_{m-1}).
\]
All earlier pairs and the offset $d$ remain unchanged. The resulting
presentation is again a generalised $T$-link presentation of $L$.
This proves the induction step.

Consequently,
\[
\begin{aligned}
   \mathcal P_{\mathbf{k}}=\bigl(
      &(r_1,s_1),\ldots,(r_{\ell-1},s_{\ell-1}),\\
      &(R_0,s+R_1),
      \bigl(R_i,R_{i+1}-R_{i-1}\bigr)_{i=1}^{n-1},\\
      &(R_n,-R_{n-1}),d
   \bigr)
\end{aligned}
\]
is a generalised $T$-link presentation representing $L$. The argument
includes $\ell=1$ and $n=1$, with the conventions for empty lists in
the statement. This proves the first assertion.

We next prove pairwise distinctness. Let $\mathbf{k}$ and
$\mathbf{k}'$ be two distinct indexing sequences. A sequence of
length $n$ produces a presentation with exactly $\ell+n$ twisting
pairs. Thus sequences of different lengths give distinct parameter
strings.

Suppose that the sequences have the same length, and let $j$ be
their first differing index. Then
\[
   R_{j-1}=R_{j-1}',
\]
whereas
\[
   R_j=k_jR_{j-1}
   \neq k_j'R_{j-1}'=R_j'.
\]
These different strand parameters occur in the same position in the
respective parameter strings. Hence
$\mathcal P_{\mathbf{k}}\neq\mathcal P_{\mathbf{k}'}$, proving the
second assertion.

Finally, Theorem~\ref{thm:generalised-t-universality} ensures that
every link in $S^3$ admits a generalised $T$-link presentation.
Starting from any such presentation and varying $k_1\geq k_0$ with
$n=1$ already gives infinitely many pairwise distinct generalised
$T$-link presentations of that link.
\end{proof}

\begin{remark}
\label{rem:markov}
Repeated Markov stabilisation gives infinitely many braid
representatives of a fixed link. Theorem~\ref{thm:main-intro}
establishes multiplicity within the prescribed generalised $T$-link
form. The construction gives explicit transformations of the
parameter strings while retaining all the restrictions in
Definition~\ref{def:generalised-t-link}.
\end{remark}

\bibliographystyle{amsplain}
\bibliography{References}

@article {BirmanKofman2009,
    AUTHOR = {Birman, Joan S. and Kofman, Ilya},
     TITLE = {A new twist on {L}orenz links},
   JOURNAL = {J. Topol.},
  FJOURNAL = {Journal of Topology},
    VOLUME = {2},
      YEAR = {2009},
    NUMBER = {2},
     PAGES = {227--248},
      ISSN = {1753-8416},
       DOI = {10.1112/jtopol/jtp007},
        MR = {2529294}
}

@misc {dePaiva2024,
    AUTHOR = {de Paiva, Thiago},
     TITLE = {{L}orenz links, {$T$}-links, minimal braids, positive braids
              with a full twist, and geometric types},
      YEAR = {2024},
      NOTE = {arXiv:2409.14824, version 2 (2026)},
       DOI = {10.48550/arXiv.2409.14824},
    EPRINT = {2409.14824},
ARCHIVEPREFIX = {arXiv},
PRIMARYCLASS = {math.GT}
}

@misc {dePaivaHuiRodriguez2024,
    AUTHOR = {de Paiva, Thiago and Hui, Connie On Yu and
              Rodr{\'i}guez Migueles, Jos{\'e} Andr{\'e}s},
     TITLE = {Generalised {$T$}-links represent all links in the
              {$3$}-sphere},
      YEAR = {2026},
      NOTE = {arXiv:2410.04391, version 3 (2026)},
       DOI = {10.48550/arXiv.2410.04391},
    EPRINT = {2410.04391},
ARCHIVEPREFIX = {arXiv},
PRIMARYCLASS = {math.GT}
}

@article {Lorenzknots,
    AUTHOR = {de Paiva, Thiago},
     TITLE = {Satellite knots that cannot be represented by positive braids with full twists},
   JOURNAL = {New York J. Math.},
  FJOURNAL = {New York Journal of Mathematics},
      YEAR = {2025},
    VOLUME = {31},
     PAGES = {1690--1701},
      ISSN = {1076-9803},
   MRCLASS = {57K10 (57K35 57M25)},
       URL = {https://nyjm.albany.edu/j/2025/31-67.html}
}

\end{document}